\documentclass[12pt]{article}

\usepackage{amssymb}
\usepackage{amsfonts}
\usepackage{amsmath}
\usepackage{multirow} 
\usepackage{titlesec}
\titleformat*{\section}{\large\bfseries}
\titleformat*{\subsection}{\normalsize\bfseries}
\titleformat*{\subsubsection}{\normalsize\bfseries}
\titleformat*{\paragraph}{\normalsize\bfseries}
\titleformat*{\subparagraph}{\normalsize\bfseries}
\DeclareMathOperator{\id}{id}
\newtheorem{theorem}{Theorem}[section]
\newtheorem{proposition}[theorem]{Proposition}

\newtheorem{remark}[theorem]{Remark}

\newtheorem{definition}[theorem]{Definition}
\newtheorem{example}[theorem]{Example}

\newenvironment{proof}[1][Proof]{\noindent\textbf{#1} }{\ \rule{0.5em}{0.5em}}
\begin{document}

\date{ }
\title{Interpolation of geometric structures compatible with a pseudo
	Riemannian metric}
\author{Edison Alberto Fern\'andez-Culma, Yamile Godoy and Marcos Salvai~\thanks{%
		Partially supported by \textsc{CONICET, FONCyT, SECyT~(UNC)}.} \\
	{\small {Universidad Nacional de C\'ordoba \,-\, Conicet}}\\
		{\small  {FaMAF\,-\,CIEM, Ciudad Universitaria, 5000 C%
			\'{o}rdoba, Argentina}}\\
	\small {\{efernandez, ygodoy, salvai\}@famaf.unc.edu.ar}}
\maketitle

\begin{abstract}
Let $(M,g)$ be a pseudo Riemannian manifold. We consider four geometric structures on $M$ compatible with $g$: two almost complex and two almost product structures satisfying additionally certain integrability conditions.
For instance, if $r$ is a product structure and symmetric with respect to $g$, then $r$ induces a pseudo Riemannian product structure on $M$. Sometimes the integrability condition is expressed by the closedness of an associated two-form: if $j$ is almost complex on $M$ and $\omega(x,y)=g(jx,y)$ is symplectic, then $M$ is almost pseudo K\"ahler.
Now, product, complex and symplectic structures on $M$ are trivial examples of generalized (para)complex structures in the sense of Hitchin. We use the latter in order to define the notion of interpolation of geometric structures compatible with $g$. We also compute the typical fibers of the twistor bundles of the new structures and give examples for $M$ a Lie group with a left invariant metric. 
\end{abstract}

\noindent 2000 MSC:\ 22F30, 22F50, 53B30, 53B35, 53C15, 53C56, 53D05

\smallskip

\noindent Keywords: generalized complex structure; interpolation; K\"{a}hler; symplectic; signature; twistor bundle

\section{Introduction}

Generalized complex geometry includes as special cases complex and
symplectic geometries. It was introduced by Nigel Hitchin in \cite{H} in
2003 and further developed by Marco Gualtieri \cite{gcg} and Gil Cavalcanti.
Since then it has greatly expanded, with a significant impact in
Mathematical Physics. The paracomplex analogue was studied by A\"{i}ssa Wade 
\cite{Wade}.

The article \cite{Salvai} deals with similar ideas, now starting not just from a
smooth manifold, but from a manifold which is already endowed with a
structure, and working out a notion of interpolation of supplementary
compatible geometric structures. More precisely, given a complex manifold $\left( M,j\right) $, six families of distinguished generalized complex or
paracomplex structures were defined and studied on $M$. Each one of them
interpolates between two geometric structures on $M$ compatible with $j$,
for instance, between totally real foliations and K\"{a}hler structures, or between hypercomplex and $\mathbb{C}$-symplectic structures. In the same article, a similar analysis was carried out for a symplectic manifold $\left( M,\omega \right) $. One had for instance the notion of a structure generalizing $\mathbb{C}$-symplectic structures and bi-Lagrangian foliations (both compatible with $\omega $).

The purpose of this note is to obtain analogous results, but now starting
from a pseudo Riemannian manifold $\left( M,g\right) $. As a by-product, we note that the search for nontrivial examples for this new structure can
bring about a better understanding of some manifolds, in the same way, for
instance, that generalized complex structures shed light on the geometry of nil- and solvmanifolds \cite{Cav,Bar}.

In Section \ref{P} we recall the definitions and properties of generalized
complex or paracomplex structures. In Subsection \ref{S2.1} we consider geometric structures on $M$ compatible with $g$, which we call integrable $\left(\lambda ,0\right) $- or $\left( 0,\ell \right) $-structures, with $\lambda,\ell =\pm 1$; for instance, $\lambda =-1$ and $\ell =-1$ give us
anti-Hermitian and almost pseudo K\"{a}hler structures, respectively. 
This nomenclature will allow us to define families of
generalized complex or paracomplex structures on $M$, called integrable $\left( \lambda ,\ell \right) $-structures, which in a certain sense,
specified in Theorem \ref{THMriemannian}, interpolate between integrable $\left( \lambda ,0\right) $- and $\left( 0,\ell \right) $-structures on $M$.
In Section \ref{SubS} we give explicitly the typical fibers of the
associated twistor bundles. The last section is devoted to examples: We give a curve of integrable $\left( 1,-1\right) $-structures on the manifold of ellipses endowed with a homogeneous Riemannian metric, joining a K\"{a}hler structure with a Riemannian product structure. Also, we find compatible integrable $\left( -1,-1\right) $-structures on certain 6-dimensional pseudo Riemannian nilmanifolds, although they do not admit the extremal integrable $\left( -1,0\right) $- and $\left( 0,-1\right) $-structures (all in the left invariant context).

\section{Preliminaries}\label{P} 

\subsection{Generalized complex and paracomplex structures}

We recall from the seminal work \cite{gcg} the definitions
and basic facts on generalized complex structures, and on generalized
paracomplex structures from \cite{Wade}. A unified approach can be found in \cite{Izu}.

Let $M$ be a smooth manifold (by smooth we mean of class $C^{\infty }$; all the objects considered will belong to this class). The extended tangent bundle is the vector bundle $\mathbb{T}M=TM\oplus TM^{\ast }$ over $M$. A canonical split pseudo Riemannian structure on $\mathbb{T}M$ is defined by 
\begin{equation}\label{b}
b\left( u+\sigma ,v+\tau \right) =\tau \left( u\right) +\sigma \left(
v\right) \text{,}
\end{equation}
for smooth sections $u+\sigma ,v+\tau $ of $\mathbb{T}M$. The \emph{Courant bracket} of these sections \cite{Courant} is given by 
\begin{equation*}
\left[ u+\sigma ,v+\tau \right] =\left[ u,v\right] +\mathcal{L}_{u}\tau -%
\mathcal{L}_{v}\sigma -\tfrac{1}{2}d\left( \tau \left( u\right) -\sigma
\left( v\right) \right) \text{,}
\end{equation*}%
where $\mathcal{L}$ denotes the Lie derivative.

A real linear isomorphism $S$ with $S^{2}=\lambda $ id, $\lambda =\pm 1$, is
called \emph{split} if tr~$S=0$ (equivalently, if the dimension of the 
$\pm \sqrt{\lambda}$-eigenspaces of $S$ coincide); this is always the case if $\lambda =-1$.

For $\lambda =\pm 1$, let $S$ be a smooth section of End\thinspace $\left( 
\mathbb{T}M\right) $ satisfying 
\begin{equation}\label{est.gen.}
S^{2}=\lambda \;\text{id, }S\text{ is split and skew-symmetric for }b
\end{equation}
and such that the set of smooth sections of the 
$\pm \sqrt{\lambda }$-eigenspace of $S$ is closed under the Courant bracket (if $\lambda =-1$,
this means as usual closedness under the $\mathbb{C}$-linear extension of
the bracket to sections of the complexification of $\mathbb{T}M$). Then, for 
$\lambda =-1$ (respectively, $\lambda =1$), $S$ is called a \emph{%
generalized complex} (respectively, \emph{generalized paracomplex})
structure on $M$. Notice that in \cite{Wade} the latter is not required to
be split.

We also have the notion of a $\left( +\right) $\emph{-generalized paracomplex}
structure $S$. It is the same as a generalized paracomplex structure, but
closedness under the Courant bracket is required only for sections of the
1-eigendistributions of $S$.

\subsection{Geometric structures compatible with a pseudo Riemannian metric}\label{S2.1}

Let $(M,g)$ be a pseudo Riemannian manifold. We consider the following
integrable geometric structures on $M$ compatible with $g$. Basically, they
are rotations and reflections in each tangent space of $M$ which are
isometries or anti-isometries for $g$ and satisfy certain integrability
conditions. The reason of the names integrable $(\lambda ,0)$- or $(0,\ell )$%
-structures will become apparent in Theorem \ref{THMriemannian}.

A \emph{product structure} on the smooth manifold $M$ is a tensor field $r$ of type 
$\left( 1,1\right) $ on $M$ with $r^{2}=$ id such that the
eigendistributions $D\left( \delta \right) $ of eigenvalues $\delta =\pm 1$
are integrable. Notice that the cases $r=\pm $ id are trivial. If $\text{dim}\,D(-1)=\text{dim}\,D(1)$, 
then the product structure $r$ is called a \emph{paracomplex structure}. 

\medskip

\noindent\textbf{Pseudo Riemannian product structure on $(M,g)$} \cite{Yano}. It is
given by a product structure $r$ on $M$ symmetric with respect to $g$ (or
equivalently, $r$ is an isometry of $g$).

Equivalently, it is given by an ordered pair of orthogonal nondegenerate foliations
(nondegenerate means that the restriction of $g$ to the leaves is
nondegenerate). The leaves of the respective foliations are the integral
submanifolds of the distributions $D\left( 1\right) $ and $D\left( -1\right) 
$. Using the properties of $r$ we have that $D\left( 1\right) $ and $D\left(
-1\right) $ are orthogonal, and hence nondegenerate, since they intersect
only at zero.

We call this structure an \textbf{integrable $\left( 1,0\right) $-structure }%
on $\left( M,g\right) $.

\medskip

\noindent\textbf{Anti-Hermitian structure on $(M,g)$} \cite{AK, Oproiu}. It is given by a complex
structure $j$ on $M$ which is symmetric with respect to $g$.

In this case, $g$ has necessarily neutral signature, since $j$ turns out to
be a linear anti-isometry of $g$ and so it sends space-like tangent vectors
to time-like tangent vectors. If $j$ is additionally parallel, then $\left(M,g,j\right) $ is called an anti-K\"{a}hler manifold.

We call this structure an \textbf{integrable $\left( -1,0\right) $-structure 
}on $\left( M,g\right) $.

\medskip

\noindent\textbf{Almost para-K\"{a}hler structure on $(M,g)$}. It is given by a
symplectic form $\omega $ on $M$ such that $r^{2}=\id$, where $r$ is the
unique tensor field satisfying $\omega (\cdot ,\cdot )=g(r\cdot ,\cdot )$.

One has that the tensor field $r$ is skew-symmetric for $g$. The
eigendistributions of $r$ are null for $g$ and Lagrangian for $\omega $,
hence $r$ is split and $g$ has neutral signature, since the dimension of a
null subspace of a Euclidean space of signature $\left( p,q\right) $ is less
than or equal to min~$\left( p,q\right) $ (Proposition 2.45 in \cite{Harvey}).
In the terminology of \cite{HarveyL} this structure would be an \textbf{%
almost K\"{a}hler }$\mathbb{D}$\textbf{-manifold.} If additionally the eigendistributions of 
$r$ are integrable, we have a bi-Lagrangian foliation \cite{bil1}.

We call this structure an \textbf{integrable $\left( 0,1\right) $-structure }%
on $\left( M,g\right) $.

\medskip

\noindent\textbf{Almost pseudo K\"{a}hler structure on $(M,g)$}. It is given by a
symplectic form $\omega $ on $M$ such that $j^{2}=-\id$, where $j$ is the
unique tensor field satisfying $\omega (\cdot ,\cdot )=g(j\cdot ,\cdot )$.

The tensor field $j$ is skew-symmetric and also a linear isometry for $g$.
Hence, $g$ must have even signature. If $j$ is additionally a complex
structure, then $\left( M,g,j\right) $ is a pseudo K\"{a}hler manifold.

We call this structure an \textbf{integrable $\left( 0,-1\right) $-structure 
}on $\left( M,g\right) $.

\medskip

\noindent\textbf{Nondegenerate foliation on $\left( M,g\right)$}.  It is given by a tensor field $r$ of type $\left(1,1\right)$ symmetric with respect to $g,$ with $r^{2}=$ id, such that $D\left( 1\right)$ is integrable. The nondegeneracy refers to the fact 
that $g$ induces a pseudo Riemannian metric on the leaves.

We call this structure an $\left( +\right) $-\textbf{integrable $\left(
1,0\right) $-structure }on $\left( M,g\right) $.

\section{Generalized geometric structures on pseudo Riemannian manifolds}\label{S2}

\subsection{Integrable $(\lambda, \ell)$-structures on $(M,g)$}

Given a bilinear form $c$ on a real vector space $V$, let $c^{\flat }\in $
End\thinspace $\left( V,V^{\ast }\right) $ be defined by $c^{\flat }\left(
u\right) \left( v\right) =c\left( u,v\right) $. The form $c$ is symmetric
(respectively, skew-symmetric) if and only if $\left( c^{\flat }\right)
^{\ast }=c^{\flat }$ (respectively, $\left( c^{\flat }\right) ^{\ast
}=-c^{\flat }$).

A \emph{paracomplex structure} $r$ \emph{on a vector bundle} $E\rightarrow M$, $E\neq TM$, over
a smooth manifold $M$ is a smooth tensor field of type $\left( 1,1\right) $
on $E$ satisfying $r^{2}=$ id and that $r_{p}$ is split for any $p\in M$.

\begin{definition}\label{Ik}
Let $\left( M,g\right) $ be a pseudo Riemannian manifold. For $k=-1$ or $k=1$, 
let $I_{k}$ be the complex, respectively paracomplex structure on $\mathbb{T}%
M$ given by 
\begin{equation*}
I_{k}=\left( 
\begin{array}{cc}
0 & k\left( g^{\flat }\right) ^{-1} \\ 
g^{\flat } & 0%
\end{array}%
\right) \text{.}
\end{equation*}
\end{definition}

\begin{remark} \label{Iksym} \emph{For further reference we verify that $I_{k}$ is symmetric
for $b$ defined in (\ref{b}). It suffices to check that $g^{\flat }\left( u\right)
\left( v\right) =g^{\flat }\left( v\right) \left( u\right) $ for all vector
fields $u,v$ on $M$ and $\tau \left( \left( g^{\flat }\right) ^{-1}\left(
\sigma \right) \right) =\sigma \left( \left( g^{\flat }\right) ^{-1}\left(
\tau \right) \right) $ for all one forms $\sigma ,\tau $ on $M$. Both
assertions follow from the symmetry of $g$ (for the second one, use the fact
that $\sigma =g^{\flat }\left( x\right)$ and $\tau =g^{\flat }\left(
y\right)$ for some vector fields $x,y$ on $M$).}
\end{remark}

Now we introduce four families of generalized geometric structures on $(M,g)$
interpolating between some of the structures listed in the previous
section.

\begin{definition}
\label{Def} Let $\left( M,g\right) $ be a pseudo Riemannian manifold.
Given $\lambda =\pm 1$ and $\ell =\pm 1$, a generalized complex structure $S$
\emph{(}for $\lambda =-1$\emph{)} or a generalized paracomplex structure $S$ 
\emph{(}for $\lambda =1$\emph{)} on $M$ is said to be an \textbf{integrable }%
$\left(\lambda ,\ell \right) $\textbf{-structure} on $\left( M,g\right)$ if 
\begin{equation}  \label{eqdef}
SI_{k}=-I_{k}S,
\end{equation}
where $k=-\lambda \ell$.
\end{definition}

We call $\mathcal{S}_{g}\left( \lambda ,\ell \right) $ the set of all
integrable $\left( \lambda ,\ell \right) $-structures on $\left( M,g\right)$.

\begin{example}\label{RQ}
\emph{If $s$ and $\omega $ are integrable $\left( \lambda ,0\right) $- and $\left(
0,\ell \right) $-structures on $\left( M,g\right) $, respectively, then easy
computations show that 
\begin{equation*}
R=\left( 
\begin{array}{cc}
s & 0 \\ 
0 & -s^{\ast }%
\end{array}%
\right) \ \ \ \ \ \ \ \ \ \text{and}\ \ \ \ \ \ \ \ \ \ Q=\left( 
\begin{array}{cc}
0 & \lambda (\omega ^{\flat })^{-1} \\ 
\omega ^{\flat } & 0%
\end{array}%
\right)
\end{equation*}%
belong to $\mathcal{S}_{g}\left( \lambda ,\ell \right) $. They are well
known to be generalized complex or paracomplex structures on $M$. In order
to verify (\ref{eqdef}) notice for instance that $s$ is symmetric for $g$ if
and only if $g^{\flat }\circ s=s^{\ast }\circ g^{\flat }$ and that $r$ and $%
j $ are as in the definitions of integrable $\left( 0,1\right) $- and $%
\left( 0,-1\right) $-structures if and only if $(\omega ^{\flat
})^{-1}\circ g^{\flat }=\ell (g^{\flat })^{-1}\circ \omega ^{\flat }$.
}
\end{example}


\begin{remark}
\emph{Analogously as in \cite{Salvai}, one can define $(+)$-integrable $(1,\ell)$-structures on $(M,g)$, for $\ell=\pm 1$.}
\end{remark}

The following simple theorem, along with Theorem \ref{fiber} below, contributes 
to render the notion of  an integrable $(\lambda, \ell)$-structure (Definition \ref{Def}) 
appropriate and relevant. 

\begin{theorem}
\label{THMriemannian} Let $\left( M,g\right) $ be a pseudo Riemannian
manifold. For $\lambda =\pm 1,\ell =\pm 1$, integrable $\left( \lambda ,\ell
\right) $-structures on $\left( M,g\right) $ interpolate between integrable $%
\left( \lambda ,0\right) $- and $\left( 0,\ell \right) $-structures on $%
\left( M,g\right) $, that is, if 
\begin{equation*}
R=\left( 
\begin{array}{cc}
s & 0 \\ 
0 & t
\end{array}
\right) \ \ \ \ \ \ \ \ \ \text{and}\ \ \ \ \ \ \ \ \ \ Q=\left( 
\begin{array}{cc}
0 & p \\ 
\omega ^{\flat } & 0
\end{array}
\right)
\end{equation*}
belong to $\mathcal{S}_{g}\left( \lambda ,\ell \right) $, then $s$ and $
\omega $ are integrable $\left( \lambda ,0\right) $- and $\left( 0,\ell
\right) $-structures on $\left( M,g\right) $, respectively.
\end{theorem}

\begin{proof}
We call $q=\omega ^{\flat }$. It is well known from \cite{gcg} and \cite%
{Wade} (see also \cite{Izu}) that if $R$ and $Q$ as above are both generalized complex
(respectively paracomplex)\ structures, then $s$ is a complex or a product
structure on $M$ and $\omega $ is a closed $2$-form. Also, that $t=-s^{\ast
} $ and $p=-q^{-1}$ (respectively, $p=q^{-1}$).

Now, since $R$ and $Q$ anti-commute with $I_{k}$ ($k=-\lambda \ell $), one
has that $g^{\flat }\circ s=s\circ g^{\flat }$ and $(\omega ^{\flat
})^{-1}\circ g^{\flat }=\ell (g^{\flat })^{-1}\circ \omega ^{\flat }$, and
so, by the equivalent statements in Example \ref{RQ}, we have that $s$ and $%
\omega $ are integrable $\left( \lambda ,0\right) $- and $\left( 0,\ell
\right) $-structures, respectively.
\end{proof}

\subsection{Integrable $(\lambda, \ell)$-structures on $(M, g)$ in classical terms}

\begin{proposition}
\label{classical} An integrable $(\lambda ,\ell )$-structure $S$ on a pseudo
Riemannian manifold $(M,g)$ has the form 
\begin{equation}
S=\left( 
\begin{array}{cc}
A & \lambda \ell B(g^{\flat })^{-1} \\ 
g^{\flat }B & -A^{\ast }%
\end{array}%
\right) \text{,}  \label{Sclassical}
\end{equation}%
where $A$ and $B$ are endomorphisms of $TM$ satisfying 
\begin{equation*}
\lambda A^{2}+\ell B^{2}=\id\text{,\ \ \ \ \ \ \ \ }AB-BA=0\text{,\ \ \ \ \
\ \ \ }g^{\flat }A=A^{\ast }g^{\flat }\text{.}
\end{equation*}
\end{proposition}

\begin{proof}
For a bilinear map $\pi :V^{\ast }\times V^{\ast}\rightarrow \mathbb{R}$,
let $\pi ^{\sharp }:V^{\ast }\rightarrow V$ be defined by $\eta ( \pi
^{\sharp }( \xi ) ) =\pi (\xi ,\eta ) $, for all $\xi ,\eta \in V^{\ast }$.

Since $S$ is a generalized complex (for $\lambda =-1$) or paracomplex
structure (for $\lambda =1$), by \cite{Crainic} (see also \cite{Izu}) one
has 
\begin{equation}
S=\left( 
\begin{array}{cc}
A & \pi ^{\sharp } \\ 
\theta ^{\flat } & -A^{\ast }%
\end{array}%
\right) \text{,}  \label{Sgeneralizada}
\end{equation}%
where $\theta $ and $\pi $ are skew-symmetric, and $A$ satisfies 
\begin{equation}
A^{2}+\pi ^{\sharp }\theta ^{\flat }=\lambda \id\text{,\ \ \ \ \ \ \ \ }%
\theta ^{\flat }A=A^{\ast }\theta ^{\flat }\text{, \ \ \ \ \ and\ \ \ \ \ \
\ }\pi ^{\sharp }A^{\ast }=A\pi ^{\sharp }\text{.}  \label{conditionGen}
\end{equation}

Now, as $S$ anti-commutes with $I_{k}$, we have that 
\begin{equation}
g^{\flat }\pi ^{\sharp }=\lambda \ell \theta ^{\flat }(g^{\flat })^{-1}\text{%
\ \ \ \ \ \ \ \ and\ \ \ \ \ \ \ \ }g^{\flat }A=A^{\ast }g^{\flat }\text{.}
\label{simplifycond}
\end{equation}

By putting $B=(g^{\flat })^{-1}\theta ^{\flat }$, we have $\pi ^{\sharp
}=\lambda \ell B(g^{\flat })^{-1}$ and so (\ref{Sclassical}) holds. Now $%
\theta ^{\flat }A=A^{\ast }\theta ^{\flat }$ and $g^{\flat }A=A^{\ast
}g^{\flat }$ yield $AB-BA=0$. Finally, from $A^{2}+\pi ^{\sharp }\theta
^{\flat }=\lambda \id$ and $\pi ^{\sharp }\theta ^{\flat }=\lambda \ell
B^{2} $, we obtain $\lambda A^{2}+\ell B^{2}=\id$.
\end{proof}

\medskip

M. Crainic obtained in \cite{Crainic} (see also \cite{Izu}) conditions on $%
A,\theta $ and $\pi $ for $S$ as in (\ref{Sgeneralizada}) to be Courant
integrable. One can deduce conditions on $A$ and $B$ as in (\ref%
{Sclassical}) for the integrability of $S$.

\subsection{A signature associated to integrable $(1,1)$-structures on $(M,g)$}

\begin{proposition}
Let $S$ be an integrable $\left( 1,1\right) $-structure on a pseudo
Riemannian manifold $\left( M,g\right) $ of dimension $m$. Then the form $%
\beta _{S}$ on $\mathbb{T}M$ defined by 
\begin{equation*}
\beta _{S}\left( x,y\right) =b\left( SI_{-1}x,y\right)
\end{equation*}
is symmetric and has signature $\left( 2n,2m-2n\right) $ for some integer $n$
with $0\leq n\leq m$.
\end{proposition}

\begin{proof} The form $\beta _{S}$ is symmetric since $S$ and $%
I_{-1}$ anti-commute and are skew-symmetric and symmetric for $b$ (see
Remark \ref{Iksym}), respectively.

We have that $\left( SI_{-1}\right) ^{2}=$ id. Let $
D(\pm 1)$ be the $\pm 1$-eigendistribution of $SI_{-1}$. We verify
that $I_{-1}\left( D(1)\right) =D(-1)$, so $D(1)$ and $D(-1)$ have both
dimension $m$.
Let $b^{\pm }=\left. b\right| _{D(\pm 1)\times
D(\pm 1)}$ and $\beta^{\pm  }=\left. \beta _{S}\right| _{D(\pm 1)\times D(\pm 1)}$. 
We compute that $\left. b\right| _{D(1)\times
D(-1)} \equiv 0$; in
particular, by the orthogonality lemma (2.30 in \cite{Harvey}), $b^{\pm }$
is nondegenerate. Suppose that $b^{+}$ has signature $\left( n,m-n\right)$. Hence $b^{-}$ has signature $\left( m-n,n\right) $ ($b$ is split). On the
other hand, we compute also that $b^{\pm  }=\pm \, \beta ^{\pm }$.
Therefore the signature of $\beta _{S}$ is $\left( 2n,2m-2n\right) $, as
desired. 
\end{proof}

\begin{definition}
An integrable $\left( 1,1\right) $-structure $S$ on $\left( M,g\right) $ as
above is called an \textbf{integrable} $\left( 1,1;n\right) $\textbf{%
-structure}, and we write \emph{sig\thinspace }$\left( S\right) =n$.
\end{definition}

The next proposition explains the meaning of sig\,$(S)$ for an extremal integrable $(1,1)$-structure on $M$ as 
in Example \ref{RQ}. 
\begin{proposition}
\label{signatura} \emph{a)} Let $r$ be an integrable $\left( 1,0\right) $%
-structure on $\left( M,g\right) $ and suppose that the metrics induced on $%
D\left( 1\right) $ and $D\left( -1\right) $ have signatures $\left(
p_{+},q_{+}\right) $ and $\left( p_{-},q_{-}\right) $, respectively. Then 
\begin{equation*}
R=\left( 
\begin{array}{cc}
r & 0 \\ 
0 & -r^{\ast }%
\end{array}%
\right)
\end{equation*}%
is an integrable $\left( 1,1;p_{+}+q_{-}\right) $-structure on $\left( M,g\right) $.

\smallskip

\emph{b)} Let $\omega $ be an integrable $\left( 0,1\right) $-structure on $%
\left( M,g\right)$. Then 
\begin{equation*}
Q=\left( 
\begin{array}{cc}
0 & \left( \omega ^{\flat }\right) ^{-1} \\ 
\omega ^{\flat } & 0%
\end{array}
\right)
\end{equation*}
is an $\left( 1,1;\frac{m}{2}\right)$-structure on $\left( M,g\right)$.
\end{proposition}

\begin{proof}
a) We compute 
\begin{equation*}
\beta _{R}\left( u+\sigma ,v+\tau \right) =-\left( g(ru,v)+\widehat{g}%
(r^{\ast }\sigma ,\tau )\right) \text{,}
\end{equation*}%
where $\widehat{g}$ is the symmetric bilinear form on $T^{\ast }M$ defined
by 
\begin{equation*}
\widehat{g}(\alpha ,\tau )=g((g^{\flat })^{-1}\alpha ,(g^{\flat })^{-1}\tau
).
\end{equation*}%
Let $h$ be the nondegenerate bilinear form on $TM$ given by $h(\cdot ,\cdot
)=g(r\cdot ,\cdot )$. Note that the signature of $\beta _{R}$ is twice the
signature of $h$, hence we compute the signature of $h$. Now, $r^{2}=\id$
and the eigendistributions $D(\pm 1)$ of $r$ are orthogonal with respect to $%
h$ and $g$. The signature of $h$ is equal to the sum of the signatures of $%
h|_{D\left( 1\right) \times D\left( 1\right) }$ and $h|_{D\left( -1\right)
\times D\left( -1\right) }$ and it can be seen that these are $\left(
p_{+},q_{+}\right) $ and $\left( q_{-},p_{-}\right) $,  respectively. Then,
the signature of $h$ is $\left( p_{+}+q_{-},q_{+}+p_{-}\right) $. Thus, sig\,$(R)=p_{+}+q_{-}$, as desired.

b) We compute 
\begin{equation*}
\beta _{Q}\left( u+\sigma ,v+\tau \right) =\widehat{g}(\omega ^{\flat
}u,\tau )-g((\omega ^{\flat })^{-1}\sigma ,v)\text{,}
\end{equation*}%
where $\widehat{g}$ is as in {a)}. Let $\{u_{1},\ldots ,u_{p},v_{1},\ldots
,v_{q}\}$ be a local orthonormal basis of $TM$ with respect to $g$ where $u_i$ 
is space-like (for $i=1,\ldots, p$) and $v_j$ is time-like (for $j=1,\ldots ,q$). 
Let $\{u_{1}^{\ast },\ldots ,u_{p}^{\ast },v_{1}^{\ast },\ldots ,v_{q}^{\ast
}\}$ be the dual basis of $T^{\ast }M$ with respect to $\omega $, i.e., $%
u_{i}^{\ast }=\omega ^{\flat }u_{i}$ and $v_{i}^{\ast }=\omega ^{\flat
}v_{i}$. It is easy to check that $\widehat{g}(\omega ^{\flat}u , \omega ^{\flat}v)=-g(u,v)$, 
for all $u,\, v\in TM$ and that 
\begin{equation*}
\{v_{1}+v_{1}^{\ast },\ldots ,v_{q}+v_{q}^{\ast },u_{1}-u_{1}^{\ast },\ldots
,u_{p}-u_{p}^{\ast },v_{1}-v_{1}^{\ast },\ldots ,v_{q}-v_{q}^{\ast
},u_{1}+u_{1}^{\ast },\ldots ,u_{p}+u_{p}^{\ast }\}
\end{equation*}%
is an orthogonal basis of $\beta _{Q}$ where the first $p+q$ vectors fields are
space-like and the remaining $p+q$ vectors fields are time-like.
\end{proof}

\section{The $\left( \protect\lambda ,\ell \right) $-twistor bundles over $\left( M,g\right) \label{SubS}$}

Let $\mathbb{L}$ denote the Lorentz numbers $a+\varepsilon b,$ $\varepsilon
^{2}=1$ and let $V$ be a vector space over $\mathbb{F}=\mathbb{R}$, $\mathbb{C}$, $\mathbb{L}$ or $\mathbb{H}$, where $\mathbb{H}=\mathbb{C}+\mathbf{j}
\mathbb{C}$ are the quaternions (we consider right vector spaces over $\mathbb{H}$). Let $O\left( m,n\right) $ be the group of automorphisms of the symmetric form of the Euclidean space of signature $\left( m,n\right) $ and
let $SO^{\ast }\left( m\right) $ be the group of automorphisms of an
anti-Hermitian form on $\mathbb{H}^{m}$ (in \cite{Harvey} it is called $%
SK\left( m,\mathbb{H}\right) $).

\begin{theorem}
\label{fiber} Let $\left( M,g\right) $ be a pseudo Riemannian manifold of
dimension $m$ and signature $\left( p,q\right) $. Then, integrable $\left(
\lambda ,\ell \right) $-structures on $\left( M,g\right) $ are smooth
sections of a fiber bundle over $M$ with typical fiber $G/H$, according to
the following table.
\begin{center}
\begin{tabular}{|c|c|c|c|c|}
\hline
$\lambda $ & $\ell $ & \emph{sig} & $G$ & $H$ \\ \hline\hline
$1$ & $1$ & n & $O\left( m,\mathbb{C}\right) $ & $O\left( n,m-n\right) $ \\ 
\hline
$1$ & $-1$ & - & $O(p,q)\times O(p,q)$ & $O(p,q)$ \\ \hline
$-1$ & $1$ & - & $O(p,p)\times O(p,p)$ & $O(p,p)$ \\ \hline
$-1$ & $-1$ & - & $O\left( m,\mathbb{C}\right) $ & $SO^{\ast }\left(
m\right) $ \\ \hline
\end{tabular}
\end{center}
(on the third row we have set $p = q$, since by Remark \ref{split} below, the existence of an integrable $(-1, 1)$-structure on $(M, g)$ forces $g$ to be split).
\end{theorem}
Before proving the theorem we introduce some notation and present a proposition. Now we work at the algebraic level. We fix $p\in M$ and call $\mathbb{E}=\mathbb{T}_{p}M$. By abuse of notation, in the rest of the
subsection we write $b$ and $I_{k}$ instead of $b_{p}$ and $(I_{k})_{p}$.

Let $\sigma \left( \lambda ,\ell \right) $ denote the set of all $S\in $ End$%
\,_{\mathbb{R}}\left( \mathbb{E}\right) $ satisfying 
\begin{equation*}
S^{2}=\lambda \id,\,S\text{ is split and skew-symmetric for }b\text{ and }
SI_{k}=-I_{k}S\text{ with }k=-\lambda \ell \text{.}
\end{equation*}
Similarly, $\sigma \left( 1,1;n\right) $ consists, by definition, of the
elements $S$ of $\sigma \left( 1,1\right) $ with sig~$\left( S\right) =n$.

Note that $\left( \mathbb{E},I_{k}\right) $ is a vector space over $\mathbb{C}$ (respectively, $\mathbb{L}$) for $k=-1$ (respectively, $k=1$) via $\left(
a+ib\right) x=ax+bI_{-1}x$ (respectively, $(a+\varepsilon b)x=ax+bI_{1}x$).

\begin{proposition}
\label{char} Let $b_{-1}:\mathbb{E}\times \mathbb{E}\rightarrow \mathbb{C}$
and $b_{1}:\mathbb{E}\times \mathbb{E}\rightarrow \mathbb{L}$ be defined by 
\begin{equation*}
b_{-1}\left( x,y\right) =b\left( x,y\right) -ib\left( x,I_{-1}y\right) \ \ \ 
\text{ and \ \ \ }b_{1}\left( x,y\right) =b\left( x,y\right) +\varepsilon
b\left( x,I_{1}y\right) .
\end{equation*}%
Then $b_{-1}$ is $\mathbb{C}$-symmetric and $b_{1}$ is $\mathbb{L}$%
-symmetric \emph{(}with respect to $I_{-1},I_{1}$, respectively\emph{)}.

Also, if $S\in \emph{End}_{\mathbb{R}}(\mathbb{E})$ satisfies $S^{2}=\lambda %
\id$ then $S\in \sigma (\lambda ,\ell )$ if and only if 
\begin{equation}
b_{k}(Sx,Sy)=-\lambda \overline{b_{k}(x,y)}  \label{condition}
\end{equation}%
for any $x,y\in \mathbb{E}$ \emph{(}where $k=-\lambda \ell \emph{)}$.
\end{proposition}

\begin{proof}
Let us denote $\epsilon _{1}=\varepsilon $ and $\epsilon _{-1}=i$ (in
particular, $\epsilon _{k}^{2}=k$). From Definition \ref{Ik} and 
the fact that $I_{k}$ is symmetric for $b$ (see Remark \ref{Iksym}) and $I_{k}^2 = k\id$ we have that, for $k=\pm 1$,
\begin{equation*}
b_{k}(x,y) = b_{k}(y,x) \,\,\mbox{ and }\,\, \epsilon _{k}b_{k}\left( x,y\right)
=b_{k}(\epsilon_{k}x , y )
\end{equation*}
for all $x,y \in \mathbb{E}$. Then the first assertion is true. 

Now we prove the second assertion. Suppose first that $S\in \sigma \left(
\lambda ,\ell \right) $. We call $k=-\lambda \ell $. Since $S$ anti-commutes
with $I_{k}$ and $S$ is skew-symmetric for $b$, we have 
\begin{eqnarray*}
b_{k}( Sx,Sy ) & = & b(Sx , Sy ) + k\epsilon _{k}b(Sx,I_{k}Sy) \\
& = &-b(S^2x,y ) - k\epsilon _{k}b\left( Sx,SI_{k}y\right) \\
& = &-b(S^2x,y ) + k\epsilon _{k}b\left( S^2x,I_{k}y\right) \\
& = &-\lambda b(x,y) + \lambda k\epsilon_{k}b(x,I_{k}y) \\
& = &-\lambda (b(x,y) - k\epsilon_{k}b(x,I_{k}y)) \\
& = &-\lambda \overline{b_{k}(x,y)} \mbox{.}
\end{eqnarray*}%
Conversely, suppose that $S^{2}=\lambda \id$ and (\ref{condition}) holds.
Note that $b_{k}(Sx,y) = - \overline{b_{k}(x,Sy)}$; it implies that 
\begin{equation} \label{skew} 
b(Sx,y) = -b(x,Sy)  \hspace{1cm}\text{and} \hspace{1cm} b(Sx,I_{k}y) = b(x,I_{k}Sy). 
\end{equation}
From the first identity in (\ref{skew}), we have that $S$ is skew-symmetric for $b$. As a
consequence of (\ref{skew}) and the nondegeneracy of $b$ we have that 
$-SI_{k} = I_{k}S$.

We complete the proof by showing that $S$ is split in the case 
$\lambda =1$ ($S^{2}=\id$). Let $D(\delta )$ be the $\delta$-eigenspaces  
of $S$, with $\delta =\pm 1$, and consider the restriction of 
$I_{k}$ to $D(1)$. This restriction gives an isomorphism between $D(1)$ and $D(-1)$
(and hence they have the same dimension). Indeed, if $v\in D(\pm 1)$, as $S$
anti-commutes with $I_{k}$, we have $S(I_{k}v)=-I_{k}(Sv)=\mp I_{k}v$.
Therefore, $S\in \sigma \left( \lambda ,\ell \right)$, as desired.
\end{proof}

\begin{remark}\label{split}
	\emph{If $(M,g)$ admits an integrable $(-1,1)$-structure then $g$ must be split. Indeed, any $S\in \sigma\left( -1,1\right) $ is an anti-isometry of the nondegenerate symmetric form $h=\text{Im~}(b_{1})$,
		whose signature is twice that of $g$. Now, $S$ sends any \textit{space-like
			subspace} of $(\mathbb{E},h)$ to a \textit{time-like subspace}, and vice
		versa, and so $h$ is split and the assertion follows.}
\end{remark}

\begin{proof} \textbf{of Theorem \ref{fiber}} We consider first the case $-\lambda \ell =k=-1$. By Proposition \ref{char}
we have that $b_{-1}$ is $\mathbb{C}$-symmetric and nondegenerate. Now, by the
Basis Theorem \cite{Harvey}, there exist complex linear coordinates $\left(
\phi _{-1}\right) ^{-1}:\left( \mathbb{E},I_{-1}\right) \rightarrow \mathbb{C%
}^{m}$ such that $B_{-1}:=\left( \phi _{-1}\right) ^{\ast }b_{-1}$ has the
form 
\begin{equation*}
B_{-1}\left( Z,W\right) =Z^{t}W\text{,}
\end{equation*}%
where $Z,W\in \mathbb{C}^{m}$ are column vectors and the superscript $t$
denotes transpose.

Let $\widetilde{\Sigma }\left( \lambda ,\ell \right) $ be the subset of End$%
\,_{\mathbb{R}}\left( \mathbb{C}^{m}\right) $ corresponding to $\sigma
\left( \lambda ,\ell \right) $ via the isomorphism $\phi _{-1}$. By the
second statement of Proposition \ref{char}, $O\left( m,\mathbb{C}\right) $
(the Lie group preserving $B_{-1}$) acts by conjugation on $\widetilde{%
\Sigma }\left( 1,1\right) $ and $\widetilde{\Sigma }\left( -1,-1\right) $.

Now, the symmetric form $B_{-1}$ above coincides with the symmetric form $%
B_{+}$ defined in the proof of Theorem 3.9 in \cite{Salvai} (except for the
inessential fact that in that theorem the dimension of $M$ is even).
Moreover, the condition (\ref{condition}) in Proposition \ref{char} is the same as the condition 
(5) in \cite[Propostion 3.11]{Salvai}. Hence $\widetilde{\Sigma }\left( 1,1;n\right) $ and $\widetilde{\Sigma }\left(
-1,-1\right)$ correspond with the sets ${\Sigma }\left( +,+;n\right) $ and 
${\Sigma }\left( -,+\right)$ given in the proof of \cite[Theorem 3.9]{Salvai}, 
respectively. Therefore, the assertions of the theorem in
these cases hold.

\smallskip

Now we analyze the remaining cases $-\lambda \ell =k=1$. We have by
Proposition \ref{char} that $b_{1}$ is $\mathbb{L}$-symmetric. We suppose as
above that the pseudo Riemannian metric $g$ on $M$ has signature $(p,q)$ and
that $\mathcal{B}=\left\{ u_{1},\dots ,u_{m}\right\} $ is an orthonormal
basis of $T_{p}M$, with $g(u_{i}, u_{i})=1$ for $1\leq i\leq p$
and $g(u_{j}, u_{j}) =-1$ for $j>p$. Since $I_{1}\left( u\right)
=g^{\flat }\left( u\right) $ for any $u\in T_{p}M$, an easy computation
shows that the matrix of $b_{1}$ with respect to $\mathcal{B}$ (thought of
as an $\mathbb{L}$-basis of $\left( \mathbb{E},I_{1}\right) $) is $
\varepsilon\, $diag\,$\left( \text{id}_{p},-\text{id}_{q}\right) $, where $\text{id}_{s}$ is the 
$s\times s$\,-identity matrix. Therefore, there exist $\mathbb{L}$%
-linear coordinates $\left( \phi _{1}\right) ^{-1}:\left( \mathbb{E}%
,I_{1}\right) \rightarrow \mathbb{L}^{m}$, such that $B_{1}:=\left( \phi
_{1}\right) ^{\ast }b_{1}$ has the form 
\begin{equation*}
B_{1}\left( (Z_{1},W_{1}),(Z_{2},W_{2})\right) =\varepsilon
(Z_{1}^{t}Z_{2}-W_{1}^{t}W_{2})\text{,}
\end{equation*}%
where $Z_{1},Z_{2}\in \mathbb{L}^{p}$ and $W_{1},W_{2}\in \mathbb{L}^{q}$.
Let $\mathbb{L}^{p,q}$ be $\mathbb{L}^{m}$ endowed with the form $B_{1}$.

Let $\widetilde{\Sigma }\left( \lambda ,\ell \right) $ be the subset of
End\thinspace $_{\mathbb{R}}\left( \mathbb{L}^{m}\right) $ corresponding to $%
\sigma \left( \lambda ,\ell \right) $ via the isomorphism $\phi _{1}$.

Let $e=\left( 1-\varepsilon \right) /2$, $\overline{e}=\left( 1+\varepsilon
\right) /2$, which are null Lorentz numbers forming an $\mathbb{R}$-basis of $%
\mathbb{L}$. One has $e^{2}=e,\overline{e}^{2}=\overline{e},e\overline{e}=0$
and $\varepsilon e=-e,\varepsilon \overline{e}=\overline{e}$. Any element of 
$\mathbb{L}^{p,q}$ can be written as $xe + y \overline{e}$ with $%
x,y\in \mathbb{R}^{m}$ and $B_{1}$ has the form 
\begin{equation*}
B_{1}\left( x_{1}e+y_{1}\overline{e},x_{2}e+y_{2}\overline{e}\right)
=\varepsilon (e\langle x_{1},x_{2}\rangle _{p,q}+\overline{e}\langle
y_{1},y_{2}\rangle _{p,q})\text{,}
\end{equation*}%
where $\langle \cdot ,\cdot \rangle _{p,q}$ is the canonical nondegenerate
symmetric bilinear form with signature $(p,q)$ on $\mathbb{R}^{m}$.

Following Section 3 of \cite{HarveyL}, the group $G$ of transformations
preserving $B_{1}$ is isomorphic to $O(p,q)\times O(p,q)$; more precisely,
by writing any element $\widehat{A}$ of $G$ in the null basis, $\widehat{A}%
=Ae+B\overline{e}$, we have that $A,B\in O(p,q)$ and 
\begin{equation}
\widehat{A}\left( xe+y\overline{e}\right) =\left( Ax\right) e+\left(
By\right) \overline{e}  \label{Atilde}
\end{equation}%
for all $x,y\in \mathbb{R}^{m}$. Clearly $G$ acts by conjugation on $%
\widetilde{\Sigma }\left( 1,-1\right) $ and $\widetilde{\Sigma }\left(
-1,1\right) $.

\medskip

\noindent\textbf{Case} $\left( 1,-1\right) $\textbf{:} Let $S\in $ End$\,_{\mathbb{R}%
}\left( \mathbb{L}^{m}\right) $ be the conjugation in $\mathbb{L}^{m}$, that
is, $S\left( xe+y\overline{e}\right) =ye+x\overline{e}$, where $x,y\in 
\mathbb{R}^{m}$. By Proposition \ref{char}, it is easy to check that $S\in 
\widetilde{\Sigma }\left( 1,-1\right) $; indeed $S^{2}=\id$ and 
\begin{eqnarray*}
B_{1}\left( S\left( x_{1}e+y_{1}\overline{e}\right) ,S\left( x_{2}e+y_{2}%
\overline{e}\right) \right) &=&B_{1}\left( y_{1}e+x_{1}\overline{e}%
,y_{2}e+x_{2}\overline{e}\right) \\
&=&\varepsilon (e\langle y_{1},y_{2}\rangle _{p,q}+\overline{e}\langle
x_{1},x_{2}\rangle _{p,q}) \\
&=&-\overline{\varepsilon }\overline{(e\langle x_{1},x_{2}\rangle _{p,q}+
\overline{e}\langle y_{1},y_{2}\rangle _{p,q})} \\
&=&-\overline{B_{1}\left( \left( x_{1}e+y_{1}\overline{e}\right) ,\left(
x_{2}e+y_{2}\overline{e}\right) \right) }\text{.}
\end{eqnarray*}%
The isotropy subgroup at $S$ of the action of $G\equiv O(p,q)\times O(p,q)$
is isomorphic to the diagonal subgroup of $O(p,q)\times O(p,q)$, which is isomorphic to $O(p,q)$.

Now, we see that the action is transitive. Let $T\in \widetilde{\Sigma }
\left( 1,-1\right) $ and let $D(\pm 1)$ be the $\pm 1$-eigenspaces of $T$. 
Let us denote $h=\text{Im~}(B_{1})$. One checks that $h$
is a real symmetric bilinear form on $\mathbb{L}^{m}$ of signature $(2p,2q)$. 
We know from the proof of Proposition \ref{char} that multiplication by 
$\varepsilon $ sends $D(1)$ to $D(-1)$ (after the identification of $\varepsilon$ with $I_{1}$
via $\phi _{1}$). Moreover, it is an isometry of $(\mathbb{L}^{m},h)$. Thus,
if we call $h^{\pm}=\left. h\right| _{D(\pm 1)\times
D(\pm 1)}$, we obtain that $\varepsilon :(D(1),h^{+})\rightarrow
(D(-1),h^{-})$ is an isometry. Therefore $h^{+}$ and $h^{-}$ have the same
signature since $D(1)$ and $D(-1)$ are orthogonal with respect to $h$ by 
(\ref{condition}). This yields that the signature of $h^{+}$ is $(p,q)$ since
the signature of $h$ is $(2p,2q)$. Let $\left\{\tilde{u}_{1},\ldots ,\tilde{u}_{m}\right\}$ 
be an orthonormal basis of $(D(1),h^{+})$, where the first $p$ vectors are space-like and the remaining 
vectors are time-like. We define the real linear transformation $F$ by $F(e_{i})=\tilde{u}_{i}$, where $\left\{ e_{i}\mid
i=1,\dots ,m\right\} $ is the canonical basis of $\mathbb{L}^{m}$. Then $F$
extends $\mathbb{L}$-linearly to a map $\tilde{F}$ preserving $B_{1}$ and
satisfying $T=\tilde{F}S\tilde{F}^{-1}$. Consequently, $\widetilde{\Sigma }
\left( 1,-1\right) $ can be identified with $O(p,q)\times O(p,q)/O(p,q)$, as
desired.

\medskip

\noindent\textbf{Case} $\left( -1,1\right) $\textbf{: }We know from Remark \ref%
{split} that $g$ is split. Let $S\in $ End$\,_{\mathbb{R}}\left( \mathbb{L%
}^{2p}\right) $ be defined by $S\left( xe+y\overline{e}\right) =r\left(
y\right) e-r\left( x\right) \overline{e}$, where $x,y\in \mathbb{R}^{2p}$
and $r\left( x_{1},x_{2}\right) =\left( x_{2},x_{1}\right) $, with $x_{i}\in 
\mathbb{R}^{p}$. It is easy to check that $S\in \widetilde{\Sigma }\left(
-1,1\right) $ and that the isotropy subgroup at $S$ of the action of $%
O(p,p)\times O(p,p)$ consists of the maps $\widehat{A}$ as in (\ref{Atilde})
with $B=rAr$. Hence, the isotropy subgroup is isomorphic to $O(p,p)$.

It remains to show that the action is transitive. Let $T\in \widetilde{%
\Sigma }\left( -1,1\right) $ and define $R=\varepsilon T$. Note that $R^{2}=%
\id$ and denote by $D(\pm 1)$ the $\pm 1$-eigenspaces of $R.$
It is easy to check that $h:=\text{Re~}(B_{1})$ is a real symmetric bilinear
form on $\mathbb{L}^{2p}$ of signature $(2p,2p)$. As before, we call 
$h^{\pm}=\left. h\right| _{D(\pm 1)\times D(\pm 1)}$. 
In the same way as in the proof of the case $\left( 1,-1\right) $, we have that $%
\varepsilon :(D(1),h^{+})\rightarrow (D(-1),h^{-})$ is an isometry.
Therefore, $h^{+}$ and $h^{-}$ have the same signature and since $D(1)$ and 
$D(-1)$ are orthogonal with respect to $h$ by (\ref{condition}), we have
that the signature of $h^{+}$ is $(p,p)$. 
Let $F$ be the real linear transformation sending an orthonormal basis of 
$(D(1),h^{+})$ to an orthonormal basis of the $1$-eigenspace of $\varepsilon
S$ (with the inner product induced from $h$). Then $F$ extends $\mathbb{L}$%
-linearly to a map $\tilde{F}$ preserving $B_{1}$ and satisfying $R=\tilde{F}%
(\varepsilon S)\tilde{F}^{-1}$. It follows that $T=\tilde{F}S\tilde{F}^{-1}$%
, and so $T\ $is conjugate to $S$ in $O\left( p,p\right) \times O\left(
p,p\right) $.
\end{proof}

\section{Examples}

We present two families of examples of left invariant integrable $(\lambda, \ell)$-structures
on Lie groups. In this case, the usually hardest computation, namely the verification of integrability, 
is simplified by the following fact (see \cite[Section 3]{Bar}). 

Let $G$ be a Lie group with Lie algebra $\mathfrak{g}$. The Courant bracket of left invariant sections of 
$\mathbb{T}G=G\times \mathfrak{g}\times \mathfrak{g}^{\ast}$ 
coincides, via the obvious identification, with the Lie bracket of the corresponding elements of the cotangent Lie algebra 
$T^{\ast }\mathfrak{g}=\mathfrak{g}\ltimes _{\text{ad}^{\ast }}\mathfrak{g}^{\ast }$, that is, the direct sum of 
$\mathfrak{g}\oplus\mathfrak{g}^{\ast}$ endowed with the Lie bracket 
\begin{equation*}
\left[ \left( x,\alpha \right) ,\left( y,\beta \right) \right] =\left( \left[
x,y\right] ,-\beta \circ \text{ad}_{x}+\alpha \circ \text{ad}_{y}\right)
\end{equation*}
for $x,y\in \mathfrak{g}$ and $\alpha ,\beta \in \mathfrak{g}^{\ast }$. So, essentially, a left invariant generalized 
(para)complex structure on a Lie group $G$ is the same as a left invariant (para)complex structure on the Lie group 
$T^{\ast }G$, which is skew-symmetric with respect to the bi-invariant canonical split metric on $T^{\ast }G$ (the proof of the assertion for the paracomplex case is similar to the proof of the complex case given in \cite{Bar}). In this way, our examples \ref{nil} and \ref{curve} below can be set in the context of the articles \cite{Bar} and \cite{Adrian}, respectively.

\subsection{Integrable $(-1,-1)$-structures on nilmanifolds}\label{nil}

Generalized complex geometry was strengthened by the existence of manifolds
admitting no known complex or symplectic structure but which do admit
generalized complex structures. We have an analogous situation in our context of integrable 
$(\lambda, \ell)$-structures on a pseudo Riemannian manifold, in the invariant setting. It
is known, for instance, that all 6-dimensional nilmanifolds admit
generalized complex structures \cite{Cav}, but some of them admit neither
complex nor symplectic left invariant structures \cite[Theorem 5.1]{salamon}. 
Nevertheless, 

\begin{theorem}
Any of the five 6-dimensional nilmanifolds admitting neither a left
invariant complex nor a left invariant symplectic structure admits a left
invariant pseudo Riemannian metric $g$ and a compatible integrable $\left( -1,-1\right) 
$-structure.
\end{theorem}

\begin{proof}
By \cite[Theorem 5.1]{salamon}, the five 6-dimensional nilpotent Lie
algebras referred to in the statement of the theorem are (with the notation
of that paper)
\begin{equation*}
\begin{array}{c}
\mathfrak{g}_{1}=(0,0,12,13,14+23,34+52),\ \ \ \ \ \ \ \mathfrak{g}%
_{2}=(0,0,12,13,14,34+52), \\ 
\mathfrak{g}_{3}=(0,0,0,12,13,14+35),\ \ \ \ \ \ \mathfrak{g}%
_{4}=(0,0,0,12,23,14+35), \\ 
\mathfrak{g}_{5}=(0,0,0,0,12,15+34).
\end{array}
\end{equation*}
Let $G_{i}$ denote the simply connected Lie group with Lie algebra $\mathfrak{g}_{i}$.
Now we endow each $\mathfrak{g}_{i}$ with a pseudo Riemannian metric $g_{i}$%
. For $\mathfrak{g}_{1}$, the Gram matrix of $g_{1}$ with respect to the
ordered basis 
\begin{equation*}
\left\{ {e_{4}},-4\,{e_{1}}+{e_{4}},-3\,{e_{1}}-{e_{2}}+{e_{3}},-{e_{1}}+{%
e_{2}}+{e_{3}},{e_{6}},2\,{e_{5}}+{e_{6}}\right\}
\end{equation*}%
is $C_{1}=$ diag$~(4,-4,2,-2,-2,2)$; in particular, the signature of $g_{1}$
is $(3,3)$. For the remaining cases we give the Gram matrix $C_{i}$ of $%
g_{i} $ with respect to the basis $\{e_{1},\ldots ,e_{6}\}$:%
\begin{eqnarray*}
C_{2} &=&{\text{diag~}(-2,1,1,1,1,1/2)},\ \ \ \ \ \ \ \ \ C_{3}={%
\text{diag~}(-1,1,1,1,1,1)}, \\
C_{4} &=&{\text{diag~}(-2,1,2,1,1,1)},  \ \ \ \ \ \ \ \ \ \ \ \ C_{5}={\text{%
diag~}(-1,1,-1,1,1,1)}.
\end{eqnarray*}

Finally, we give compatible integrable $\left( -1,-1\right) $-structures $S_{i}$ on $(G_{i}, g_{i})$
in classical terms, as in (\ref{Sclassical}), that is, we make explicit the 
$6\times 6$\,-matrices $A^{i}$ and $B^{i}$ in each case with
respect to the basis $\{e_{1},\ldots ,e_{6}\}$ (we write down only their nonzero
components).


\begin{table}[htb] 
\centering
\begin{tabular}{|p{0.25cm}|c|c|}
\hline
\multirow{2}{3cm}{$\mathfrak{g}_1$} 
&
$A_{12}^{1}=A_{43}^{1}=A_{44}^{1}=A_{55}^{1}=A_{65}^{1}=1,$ 

&
$B_{35}^{1}=B_{46}^{1}=B_{52}^{1}=B_{62}^{1}=-1,$ 
\\ 
&
$A_{21}^{1}=A_{32}^{1}=A_{33}^{1}=A_{66}^{1}=-1$
&
$B_{36}^{1}=B_{53}^{1}=B_{64}^{1}=2, B_{51}^{1}=1, B_{54}^{1}=4$
\\
\hline
\multirow{2}{3cm}{$\mathfrak{g}_2$}
&
$A_{11}^{2}=A_{33}^{2}=A_{66}^{2}=1, A_{12}^{2}=-1$,
&
$B_{36}^{2}=B_{45}^{2}=-1, B_{54}^{2}=1, B_{63}^{2}=2$
\\ 
&$A_{22}^{2}=-1, A_{21}^{2}=2$ 
&
\\
\hline
{$\mathfrak{g}_3$}
&
$A_{12}^{3}=-1, A_{21}^{3}=1$
&
$B_{36}^{3}=B_{45}^{3}=-1, B_{54}^{3}=B_{63}^{3}=1$
\\ 
\hline
\multirow{2}{3cm}{$\mathfrak{g}_4$}
&
$A_{11}^{4}=1, A_{12}^{4}=A_{22}^{4}=A_{33}^{4}=-1$,
&
$B_{36}^{4}=B_{45}^{4}=-1, B_{54}^{4}=1, B_{63}^{4}=2$
\\ 
&$A_{66}^{4}=-1, A_{21}^{4}=2$
&\\ 
\hline
{$\mathfrak{g}_5$}
&
$A_{12}^{5}=A_{43}^{5}=1, A_{21}^{5}=A_{34}^{5}=-1$
&
$B_{56}^{5}=-1,\ \ \ \ \ B_{65}^{5}=1$
\\ 
\hline
\end{tabular}
\end{table}

One can check that $S_i $ satisfies (\ref{est.gen.}) for $\lambda=-1$ and (\ref{eqdef}) for $\kappa=-1$. 
It remains to check that $S_i$ is integrable. According to the fact above the statement of the theorem, 
we give the cotangent algebra $T^{\ast }\mathfrak{g}_{i}$ by specifying the Lie
bracket on the basis $\left\{ e_{1},\dots ,e_{6},e_{1}^{\ast },\dots
,e_{6}^{\ast }\right\} $, where $e_{i}^{\ast }\left( e_{j}\right) =\delta
_{ij}$. Since $\mathfrak{g}_{i}$ is a subalgebra of $T^{\ast }\mathfrak{g}_{i}$ and $%
\mathfrak{g}_{i}^{\ast }$ is an abelian subalgebra of $T^{\ast }\mathfrak{g}_{i}$,
it suffices to compute $\left[ e_{i},e_{j}^{\ast }\right] $ for $1\leq
i,j\leq 6$. We list them following the notation above (for instance, $%
5^{\ast }1+6^{\ast }3$ in the forth place means that $\left[ e_{5}^{\ast
},e_{1}\right] =\left[ e_{6}^{\ast },e_{3}\right] =e_{4}^{\ast }$):%
\begin{eqnarray*}
\text{For }\mathfrak{g}_{1} &:&\left( 23^{\ast }+34^{\ast }+45^{\ast
},3^{\ast }1+35^{\ast }+6^{\ast }5,4^{\ast }1+5^{\ast }2+46^{\ast },5^{\ast
}1+6^{\ast }3,26^{\ast },0\right) \text{;} \\
\text{For }\mathfrak{g}_{2} &:&\left( 23^{\ast }+34^{\ast }+45^{\ast
},3^{\ast }1+6^{\ast }5,4^{\ast }1+46^{\ast },5^{\ast }1+6^{\ast }3,26^{\ast
},0\right) \text{;} \\
\text{For }\mathfrak{g}_{3} &:&\left( 24^{\ast }+35^{\ast }+46^{\ast
},4^{\ast }1,5^{\ast }1+56^{\ast },6^{\ast }1,6^{\ast }3,0\right) \text{;} \\
\text{For }\mathfrak{g}_{4} &:&\left( 24^{\ast }+46^{\ast },4^{\ast
}1+35^{\ast },5^{\ast }2+56^{\ast },6^{\ast }1,6^{\ast }3,0\right) \text{;}
\\
\text{For }\mathfrak{g}_{5} &:&\left( 25^{\ast }+56^{\ast },5^{\ast
}1,46^{\ast },6^{\ast }3,6^{\ast }1,0\right) \text{.}
\end{eqnarray*}
Using this, one computes that the Nijenhuis tensor of $S_i$ vanishes. Hence $S_i$ is integrable. 

Notice that in general $B^{i}$ does not have maximal rank. This asserts that, in a certain sense, 
$S_i$ is far away from the extremal cases as in Example \ref{RQ}.
\end{proof}

\bigskip

\subsection{A curve of integrable $\left( 1,-1\right) $-structures on the Riemannian manifold of ellipses}\label{curve}

\smallskip

Although the theory of integrable $\left( \lambda ,\ell \right)$-structures
applies mostly to the pseudo Riemannian case (and even more to neutral
signature) we have the following example in the Riemannian setting.

Given $a\in \mathbb{R}^{2}$ and a positive definite $2\times 2$\,-matrix $A$ with 
$\det A=1$, let 
\begin{equation*}
E\left( A,a\right) =\left\{ z\in \mathbb{R}^{2}\mid \left( z-a\right)
^{t}A^{-1}\left( z-a\right) =1\right\} \text{,}
\end{equation*}%
which is an ellipse enclosing a region of area $\pi $ (centered at the point 
$a$, with axes not necessarily parallel to the coordinate axes). Any such
ellipse has this form and we call $\mathcal{E}$ the space consisting of all
of them.

The group $H=SL_{2}\left( \mathbb{R}\right) \ltimes \mathbb{R}^{2}$ of all
area preserving affine transformations of the plane acts canonically and
transitively on $\mathcal{E}$, with isotropy subgroup $K=SO_{2}\times
\left\{ 0\right\} $ at $E\left( I,0\right) $, the circle of radius $1$
centered at the origin (in fact, $\left( A^{1/2},a\right) $ sends $E\left(
I,0\right) $ to $E\left( A,a\right) $). Thus, we can identify $\mathcal{E}%
=H/K$ and so $\mathcal{E}$ is a four dimensional manifold. Consider on $H$
the left invariant Riemannian metric $g$ given at the identity by 
\begin{equation*}
g\left( \left( X,x\right) ,\left( Y,y\right) \right) =\text{tr}~\left(
X^{t}Y\right) +x^{t}y\text{.}
\end{equation*}
There is an $H$-invariant Riemannian metric $g$ on $\mathcal{E}$ such that the canonical 
projection $H \rightarrow \mathcal{E}$ is a Riemannian submersion. This metric $g$ on $\mathcal{E}$ 
turns out to be isometric to an irreducible left invariant Riemannian metric on a solvable Lie group 
$G$ whose metric Lie algebra has an
orthonormal basis $\mathcal{B}=\left\{ e_{1},e_{2},e_{3},e_{4}\right\} $
such that 
\begin{equation*}
[e_{1},e_{2}]=[e_{1},e_{4}]=0,\,[e_{3},e_{4}]=2e_{4},%
\,[e_{4},e_{2}]=2e_{1},[e_{3},e_{1}]=e_{1},\,[e_{2},e_{3}]=e_{2}
\end{equation*}
(see \cite[Example 3, page 19]{Apos}). One can take $G=T_{1}\ltimes \mathbb{R}^{2}$, where $T_{1}$ is the group of upper 
triangular $2 \times 2$\,-matrices with determinant 1, 
$$
e_{1}= (0, (1,0)), \hspace{0.25cm} e_{2}= (0, (0,1)), \hspace{0.25cm}
e_{3}= \left(\left( 
\begin{array}{cc}
1 & 0 \\ 
0 & -1
\end{array}
\right) , 0\right) \ \ \text{and}\ \ e_{4}= \left(\left( 
\begin{array}{cc}
0 & 2 \\ 
0 & 0
\end{array}
\right) , 0\right).
$$
(Notice that the orthogonal projection of 
$\left( 
\begin{array}{cc}
0 & 2 \\ 
0 & 0
\end{array}
\right)$ onto $SO_{2}(\mathbb{R})^{\bot}$ is 
$\left( 
\begin{array}{cc}
0 & 1 \\ 
1 & 0
\end{array}
\right)$.)
We will consider three geometric structures on $G$ compatible with the
Riemannian metric: The first one, a Riemannian product structure $r$ whose
eigendistributions are $D(1)=\text{span}\{e_{1},e_{2}\}$ and $D(-1)=\text{%
span}\{e_{3},e_{4}\}$ (both are subalgebras and the restriction of $g$ to
them is trivially nondegenerate).

Now, we recall from Example 3 and Lemma 1 in \cite{Apos} the definition of two 
left invariant almost complex structures $J_{+}$ and $J_{-}$ on $G$ given by 
$J_{\varepsilon }(e_{1})=e_{2}$, $J_{\varepsilon }(e_{3})=\varepsilon e_{4}$ 
($\varepsilon =\pm 1$). It tuns out that $\left( G,g,J_{\varepsilon }\right) $ 
is a K\"ahler manifold for $\varepsilon =1$ and an almost K\"ahler manifold for $\varepsilon =-1$, which
is not K\"ahler (for the definition of $J_{-}$ in terms of 
$J_{+}$ one computes that $D\left( 1\right) $ and $D\left( -1\right) $ are
the eigenspaces of the Ricci tensor of $\mathcal{E}$, with eigenvalues $0$
and $-6$, respectively). We comment that $\left( G,g,J_{-}\right) $ is 
isomorphic to the unique proper 3-symmetric space in four dimension 
(see \cite{Apos, Gray, Kowalski}).

Calling $\omega _{\varepsilon }$ the symplectic form associated to $%
J_{\varepsilon }$ and $g$, according to our nomenclature, $\omega
_{\varepsilon }$ is then an integrable $\left( 0,-1\right) $-structure on $G$%
. Let $R$, $Q_{\varepsilon }$ ($\varepsilon =\pm 1$) be the integrable $%
(1,-1)$-structures on $(G,g)$ associated with $r$ and $\omega _{\varepsilon }
$, as in Example \ref{RQ}, respectively. The following proposition presents
a curve $\Phi _{\varepsilon }$ of integrable $\left( 1,-1\right) $%
-structures on $G$ joining $r$ with $\omega _{\varepsilon }$.

\begin{proposition}
For any $t\in \mathbb{R}$, 
\begin{equation*}
\Phi _{\varepsilon }(t)=\cos t~R+\sin t~Q_{\varepsilon }
\end{equation*}%
defines a left invariant integrable $(1,-1)$-structure on $(\mathcal{E},g)$. 
\end{proposition}

\begin{proof}
Since $R$ and $Q_{\varepsilon}$ are integrable $(-1,1)$-structures on $(G,g)$, they 
satisfy (\ref{est.gen.}) for $\lambda=1$ and (\ref{eqdef}) for $\kappa=1$. This implies 
easily that $Q_{\varepsilon} (t)$ satisfies those conditions as well ($R$ and $Q_{\varepsilon}$ anticommute).
It remains only to check that the eigenspaces $D_{\varepsilon }\left( \delta \right)$ ($\delta=\pm 1$) are Courant
involutive, or equivalently, by the fact stated at the beginning of the section, that they are Lie subalgebras of 
$T^{\ast }\mathfrak{g}$. For this, we consider the basis $\mathcal{C}$ of the cotangent
algebra given by the juxtaposition of $\mathcal{B}$ with its dual basis  $
\{e_{1}^{\ast },e_{2}^{\ast },e_{3}^{\ast },e_{4}^{\ast }\}$. As in the
first example in this section, we have to compute only $%
[e_{i},e_{j}^{\ast }]$, for $1\leq i,j\leq 4$. In this case, we obtain 
\begin{equation*}
\lbrack e_{1},e_{1}^{\ast }]=e_{3}^{\ast },\,[e_{2},e_{1}^{\ast
}]=2e_{4}^{\ast },\,[e_{2},e_{2}^{\ast }]=-e_{3}^{\ast },[e_{3},e_{1}^{\ast
}]=-e_{1}^{\ast },
\end{equation*}
\begin{equation*}
[e_{3},e_{2}^{\ast }]=e_{2}^{\ast },\, \lbrack e_{3},e_{4}^{\ast }]=-2e_{4}^{\ast },[e_{4},e_{1}^{\ast
}]=-2e_{2}^{\ast },\,[e_{4},e_{4}^{\ast }]=2e_{3}^{\ast }\text{.}
\end{equation*}%
The matrix of $R$ with respect to $\mathcal{C}$ is diag~$%
(1,1,-1,-1,-1,-1,1,1)$ and the matrix of $Q_{\varepsilon }$ with respect to
the same basis is$\,$%
\begin{equation*}
\left( 
\begin{array}{cc}
0_{4} & -B_{\varepsilon } \\ 
B_{\varepsilon} & 0_{4}%
\end{array}%
\right) 
\end{equation*}%
where%
\begin{equation*}
B_{\varepsilon }=\text{diag}~(j,\varepsilon j)\ \ \ \ \ \ \ \ \ \ \ \ \text{%
with\ \ \ \ \ \ \ \ \ \ \ }j=\left( 
\begin{array}{cc}
0 & -1 \\ 
1 & 0%
\end{array}%
\right) \text{,}
\end{equation*}%
and $0_{4}$ denotes the $4\times 4$ zero matrix. We have that the eigenspace $D_{\varepsilon }\left( \delta \right) $
is spanned by the vectors:
\begin{eqnarray*}
&&-(\sin t )\,e_{1}+(\cos \,t -\delta )e_{2}^{\ast },\ \ \ \ \ \ \ \ \ \ \ \ (\sin t) \,e_{2}+(\cos t\, -\delta
)e_{1}^{\ast },\  \\
&&\ \varepsilon (\sin t) \, e_{3}+(\cos t \,+\delta )e_{4}^{\ast },\ \ \ \ \ \ \ \ \ -\varepsilon
(\sin t) \,e_{4}+(\cos t \,+\delta )e_{3}^{\ast }\text{.}
\end{eqnarray*}
An easy computation shows that any of these subspaces is a subalgebra of the cotangent algebra $T^{\ast }\mathfrak{g}$, 
as desired.
\end{proof}

\end{document}